\documentclass[11pt]{amsart}
\usepackage[utf8]{inputenc}

\usepackage{amsmath}%
\usepackage{amsfonts}%
\usepackage{amssymb}%
\usepackage{amsthm}
\usepackage{bm}
\usepackage{graphicx}
\usepackage{stmaryrd}
\usepackage{centernot} 
\usepackage{tikz}
\usepackage{wrapfig}
\usepackage{mathtools}

\usepackage{url}
\usepackage{geometry}
\usepackage{enumerate}

\theoremstyle{plain}
\newtheorem{theorem}{Theorem}[section]

\newtheorem{lemma}[theorem]{Lemma}

\newtheorem{claim}[theorem]{Claim}
\newtheorem{proposition}[theorem]{Proposition}
\newtheorem{problem}{Problem}

\theoremstyle{definition}
\newtheorem{definition}[theorem]{Definition}
\newtheorem{example}[theorem]{Example}

\title{Every CBER is smooth below the Carlson-Simpson generic partition}

\author{Aristotelis Panagiotopoulos}
\address{Department of Mathematical Sciences, Carnegie Mellon University}
\email{aristotelis.panagiotopoulos@gmail.com}

\author{Allison Wang}
\address{Department of Mathematical Sciences, Carnegie Mellon University}
\email{ayw2@andrew.cmu.edu}

\begin{document}
\maketitle

\begin{abstract}
Let $E$ be a countable Borel equivalence relation on the space $\mathcal{E}_{\infty}$ of all infinite partitions of the natural numbers. We show that $E$ coincides with equality below a Carlson-Simpson generic element of $\mathcal{E}_{\infty}$. In contrast, we show that there is a hypersmooth equivalence relation on $\mathcal{E}_{\infty}$ which is Borel bireducible with $E_1$ on every Carlson-Simpson cube. Our arguments are  classical and require no background in forcing.
\end{abstract}

\section{Introduction}

One of the prominent ongoing research programs of descriptive set theory concerns the complexity theory of classification problems. Formally, a classification problem is any pair $(X,E)$ where $X$ is a Polish space and $E$ is any equivalence relation on $X$ which is analytic as a subset of $X\times X$. A classification problem $(X,E)$ is considered to be of less or equal complexity to the classification problem $(Y,F)$ if there is a {\bf Borel reduction} from $E$ to $F$, \textit{i.e.}, if there is a Borel map $f\colon X\to Y$ so that $x E x' \iff f(x) F f(x')$ for all $x,x'\in X$. We write $(X,E)\leq_B (Y,F)$, or simply $E\leq_B F$, when $E$ is Borel reducible to $F$. In the lower part of the Borel reduction hierarchy one finds the complexity class of all Countable Borel Equivalence Relations (CBERs).

\subsection{Countable Borel Equivalence Relations} The class of CBERs consists of all classification problems $(X,E)$ so that $E$ is Borel and each $E$-class $[x]_E\subseteq X$ is countable. The class of CBERs has received significant attention over the years, partially due to its connections to ergodic theory. On the one hand, measure-preserving actions $\Gamma\curvearrowright X$ of countable discrete groups $\Gamma$ naturally give rise to CBERs by setting $xEx' \iff \exists \gamma \in \Gamma (\gamma x=x')$. On the other hand, it seems that the only methods we currently have at our disposal for discerning distinct complexities within the class of CBERs make essential use of measure-theoretic arguments.

To elaborate a bit more on the latter point, we recall some facts about the structure of the quasi-ordering $\leq_B$ within the class of CBERs; for a comprehensive survey, see \cite{CBER}. The $\leq_B$-least CBER is equality $=$ on $2^{\mathbb{N}}$. The $\leq_B$-successor of $(2^{\mathbb{N}},=)$ is  $(2^{\mathbb{N}},E_0)$, where $E_0$ is eventual equality: $x E_0 y$ if there is $n_0\in \mathbb{N}$ so that for all $n\geq n_0$ we have  $x(n)=y(n)$. We say that  $(X,E)$ is  {\bf smooth} if  $(X,E)\leq_B (2^{\mathbb{N}},=)$. We say that it is {\bf hyperfinite} if $(X,E)\leq_B (2^{\mathbb{N}},E_0)$. The class of CBERs admits a $\leq_B$-top element $E_{\infty}$. This is the orbit equivalence relation of the Bernoulli shift $F_2\curvearrowright 2^{F_2}$ of the free group with $2$ generators. We say that $(X,E)$ is {\bf intermediate} if $E_0<_B E <_B E_{\infty}$.

The structure of intermediate CBERs is very complex. For example, the partial order of Borel subsets of the reals under inclusion embeds in the quasi-order  induced by  $\leq_B$ on CBERs; see  
\cite{AK}. However, all known proofs of existence of intermediate CBERs, as well as all known arguments for comparing intermediate CBERs with respect to $\leq_B$, use measure-theoretic methods and cocycle superrigidity techniques of ergodic theory. A partial explanation as to why this is the case comes from a result of Hjorth and Kechris, which states that every CBER $E$ on a Polish space $X$ is hyperfinite when restricted to an appropriately chosen $E$-invariant  comeager $C\subseteq X$; see \cite[Theorem 6.2]{HK}. In other words, the intrinsic complexity of a CBER sitting above $E_0$ hides within a negligible part of $X$ from the point of view of Baire category. This renders Baire-category methods too crude for the fine structure of the complexity class of CBERs.

Given that the last four decades have generated several open problems about CBERs that remain open to this day (\textit{e.g.}, Martin's conjecture, see  \cite[Problem 11.5]{CBER}; or Weiss' problem on amenable actions, see \cite[Problem 7.27]{CBER}), 
one would like to have methods aside from measure for discerning complexities beyond $E_0$. In particular, one would like to know which notions of ``largeness,'' alternative to measure or Baire category, interact non-trivially with the intrinsic complexity of intermediate CBERs. To make this precise, we recall some definitions from \cite{KSZ}. Let $\mathcal{I}$ be a $\sigma$-ideal on a Polish space and let $E$ be an analytic equivalence relation on $X$. We say that {\bf $E$ is in the spectrum of $\mathcal{I}$} if for every Borel $B\subseteq X$ with $B\not\in \mathcal{I}$ there is a   Borel $C\subseteq B$ with $C\not\in \mathcal{I}$ so that $E$ and $E\upharpoonright C$ are Borel bireducible. On the other hand, we say that $\mathcal{I}$ {\bf canonizes $E$} to some $F\leq_B E$, if for every Borel $B\subseteq X$ with $B\not\in \mathcal{I}$ there is a   Borel $C\subseteq B$ with $C\not\in \mathcal{I}$ so that $E \upharpoonright C \leq_B F$. In this terminology, all known intermediate CBERs are in the spectrum of the $\sigma$-ideal $\mathcal{I}_{\mu}$ of $\mu$-null sets for some appropriately chosen $\sigma$-additive measure, while the aforementioned result \cite[Theorem 6.2]{HK} shows that the $\sigma$-ideal $\mathcal{I}_*(X)$ of all meager sets of a Polish space $X$ canonizes every CBER to $E_0$ in a very strong way.

\subsection{Ramsey notions of ``largeness''} Ramsey theory provides a plethora of $\sigma$-ideals whose relationship with the complexity of CBERs is not yet well understood. Recall that every topological Ramsey space $(\mathcal{R},\leq ,r)$ comes together with two topologies: its Polish topology and its Ellentuck topology; see \cite{Ramsey}. We denote by $\mathcal{I}(\mathcal{R})$ the collection of all Ellentuck-meager subsets of $\mathcal{R}$. By Borel or analytic subsets of $\mathcal{R}$ or $\mathcal{R}^2$  we mean Borel or analytic in the Polish topology of $\mathcal{R}$.
For the classical Ramsey space $[\mathbb{N}]^{\mathbb{N}}$ of all infinite subsets of $\mathbb{N}$, classical results of Mathias  \cite{Mathias} regarding the extensions of a model of set theory by Mathias-generic reals show that $\mathcal{I}([\mathbb{N}]^{\mathbb{N}})$ canonizes every CBER to $E_0$; see \cite[Theorem 8.17]{KSZ}. Similar results hold for the Milliken space $(\mathbb{N})^{\mathbb{N}}$ of  infinite sequences $(A_n)_n$ of finite sets of natural numbers with the property that $\min(A_{n+1}) > \max(A_n)$ for every $n\in\mathbb{N}$, where $(A_n)_n\leq (B_n)_n$ if for all $m\in\mathbb{N}$ we have that  $A_m$ is a union of elements of $(B_n)_n$. By \cite[Theorem 8.28]{KSZ},  $\mathcal{I}((\mathbb{N})^{\mathbb{N}})$  canonizes every CBER to $E_0$ as well.

It turns out that there exist Ramsey spaces which canonize CBERs to $(2^{\mathbb{N}},=)$. In fact, all Ramsey spaces which have previously been shown to have this property satisfy a much stronger canonization property: they canonize {\bf any analytic equivalence relation} to $(2^{\mathbb{N}},=)$. The earliest such example is the Halpern--L\"auchli--Milliken space $\mathcal{S}_{\infty}(2^{<\mathbb{N}})$ of strong subtrees of the complete binary tree $2^{<\mathbb{N}}$; see  \cite[Theorem 8.1]{KSZ}. The second known example of a Ramsey-theoretic ideal with similarly strong canonization properties comes from the Carlson-Simpson dual Ramsey theorem  \cite{CS}  for finite partitions of the natural numbers. For every positive $k\in\mathbb{N}$ let $\mathcal{E}_k$ be the collection of all equivalence relations $A$ on $\mathbb{N}$ with exactly $k$-many classes. Let $\mathcal{E}_{\infty}$ be the collection of all equivalence relations $A$ on $\mathbb{N}$ with infinitely many classes. We write $A\leq B$ if the equivalence relation $A$ is a coarsening of $B$; see Section \ref{Section:Def} for definitions. For every $\kappa\in \{1,2,\ldots\}\cup\{\infty\}$ one can define a $\sigma$-ideal $\mathcal{I}(\mathcal{E}_{\kappa})$ on $\mathcal{E}_{\kappa}$ as follows. First, for every $m\in\mathbb{N}$ and $B\in\mathcal{E}_{\infty}$, let $[m,B]_{\kappa}$ be the set of all $A\in\mathcal{E}_{\kappa}$ so that $A\leq B$ and $A,B$ agree on $\{0,\ldots,m-1\}$. Now suppose $\mathcal{X}\subseteq \mathcal{E}_{\kappa}$. We set $\mathcal{X}\in\mathcal{I}(\mathcal{E}_{\kappa})$ if for every $B\in\mathcal{E}_{\infty}$ and $m\in\mathbb{N}$, there is $B'\in [m,B]_{\infty}$  so that for all $A\in [m,B']_{\kappa}$ we have that $A\not\in \mathcal{X}$; see  \cite[Section 5.6]{Ramsey}. In \cite{Doucha} it is shown that if $\kappa$ is finite, then $\mathcal{I}(\mathcal{E}_{\kappa})$ canonizes any analytic equivalence relation to $(2^{\mathbb{N}},=)$.

\subsection{Main results}
In this paper we study the canonization properties of the $\sigma$-ideal $\mathcal{I}(\mathcal{E}_{\infty})$ associated to the topological Ramsey space $\mathcal{E}_{\infty}$ of infinite partitions of $\mathbb{N}$. We first show that, unlike $\mathcal{E}_{\kappa}$ with $\kappa$ finite, there exist  non-hyperfinite Borel equivalence relations which maintain their complexity on every  $\mathcal{I}(\mathcal{E}_{\infty})$-positive subset of $\mathcal{E}_{\infty}$.
Recall that if $E_1$ denotes the equivalence relation of eventual equality of sequences of reals, then $E_0\lneq_B E_1$.

\begin{proposition}\label{P:1}
$E_1$ is in the spectrum of $\mathcal{I}(\mathcal{E}_{\infty})$. More precisely, there exists a Borel equivalence relation on $\mathcal{E}_{\infty}$ whose restriction on every set of the form $[m,B]_{\infty}\subseteq\mathcal{E}_{\infty}$ is Borel bireducible with $E_1$. \end{proposition}

On the other hand, the content of our main theorem is that when it comes to CBERs, $\mathcal{E}_{\infty}$ shares the same canonization properties as $\mathcal{S}_{\infty}(2^{<\mathbb{N}})$ and  $\mathcal{E}_{\kappa}$ with $\kappa$ finite. For the statement below, recall that $\mathcal{D}\subseteq \mathcal{E}_{\infty}$ is \emph{Ramsey conull} if $\mathcal{E}_{\infty}\setminus \mathcal{D} \in \mathcal{I}(\mathcal{E}_{\infty})$.

\begin{theorem}\label{T:Main}
Let $E$ be a countable Borel equivalence relation on  $\mathcal{E}_{\infty}$. Then there exists a Ramsey conull $\mathcal{D}\subseteq \mathcal{E}_{\infty}$ so that for all $B\in\mathcal{D}$ and $A, A' \leq B$ we have that
\[A \; E \; A' \iff A = A'.\]
In particular,  $\mathcal{I}(\mathcal{E}_{\infty})$  canonizes every CBER to $(2^{\mathbb{N}},=)$. 
\end{theorem}

In contrast to all the previously known results  regarding the canonization properties of Ramsey--type ideals \cite{KSZ,Doucha,Mathias}, our arguments here are entirely classical/combinatorial and require no background in set-theoretic forcing. They also  differ in an essential way from earlier arguments for ``canonization to $=$'',
and not merely on the level of language.  Indeed, Proposition \ref{P:1} illustrates that Theorem \ref{T:Main} cannot be attained by the ``mutual generics" technique for (reduced) product forcings that was used in \cite[Theorem 8.1]{KSZ} and \cite[Theorem 0.1]{Doucha}; finer techniques designed specifically for CBERs are necessary. In the process of developing such techniques, we prove Lemma \ref{L:Tracial}, which is interesting in its own right since it seems that it could serve as the starting point for similar canonization results for other topological Ramsey spaces. As an example, in Section \ref{S:Last} we illustrate how our arguments adapt  to prove  \cite[Theorem 8.1]{KSZ} for the special case of CBERs.

\subsection{Some open problems} All topological Ramsey spaces that we know of canonize every CBER to either $(2^{\mathbb{N}},E_0)$ or $(2^{\mathbb{N}},=)$. This raises the following two natural questions:

\begin{problem}
Characterize all topological Ramsey spaces that canonize every CBER to $=$.
\end{problem}

\begin{problem}
Is there a topological Ramsey space $\mathcal{R}$ and some CBER $E$ on $\mathcal{R}$ so that:
\begin{enumerate}
    \item $E_0\lneq E$; and 
    \item $E$ is in the spectrum of $\mathcal{I}(\mathcal{R})$?
\end{enumerate}
\end{problem}

Recall that an equivalence relation $E$ on a  Polish space $X$ is hypersmooth if $E = \bigcup_{n\in\mathbb{N}}E_n$, where $E_n$ are smooth equivalence relations on $X$ and $E_n\subseteq E_{n+1}$ for all $n$. Another natural question that is raised by Proposition \ref{P:1} and the general theory of hypersmooth equivalence relations \cite{KL} is the following:

\begin{problem}
Let $E$ be an analytic equivalence relation on $\mathcal{E}_{\infty}$. Does $\mathcal{I}(\mathcal{E}_{\infty})$ canonizes $E$ to some hypersmooth equivalence relation?
\end{problem}

\section{Acknowledgments}
We would like to thank Andrew Marks for making us aware of  Mathias' results from \cite{Mathias}. The results we present here are just a byproduct of going down this rabbit hole. We are also grateful to M. Karamanlis, A.S. Kechris, N. Dobrinen, C. Conley, and J. Cummings whose genuine interest in this project motivated us to keep propelling down the rabbit hole.

The second author was supported by the National Science Foundation Graduate Research Fellowship Program under Grant Nos. DGE1745016 and DGE2140739, and by the ARCS Foundation. 

\section{Background}\label{Section:Def}

As usual, natural numbers $n\in\mathbb{N}$ are identified with the set $\{0,\ldots,n-1\}$ of their predecessors. If $R$ is a set of ordered pairs $(x,y)$ then $\mathrm{dom}(R)=\{x \colon \exists y  \; (x,y)\in R \}$.
Let $\mathcal{E}_{\infty}$ be the collection of all equivalence relations $E$ on $\mathbb{N}$ that have an infinite quotient $\mathbb{N}/E$. We will think of $E$ as a partition of $\mathbb{N}$ into infinitely many equivalence classes. Let  $\mu(E):=\{\mathrm{min}([x]_E) \colon x \in \mathbb{N} \}$ be the collection of all minimum representatives of the equivalence classes $[x]_E$ of $E$, and let $\{\mu_n(E)\colon n\in \mathbb{N}\}$ be the increasing enumeration of $\mu(E)$. Notice that $0=\mu_0(E)\in \mu (E)$ for all $E\in \mathcal{E}_{\infty}$. 

Let $A,B\in \mathcal{E}_{\infty}$. We say that $A$ is coarser than $B$ and write $A\leq B$ if every $A$-class is a union of $B$-classes. The $n$-th approximation $r_n(A)$ of $A\in\mathcal{E}_{\infty}$ is the equivalence relation $A\upharpoonright \mu_n(A)$ that is the restriction of $A$ to the finite set $\{0,\ldots,\mu_n(A)-1\}$. Notice that $r_0(A)=\emptyset$ for all $A\in\mathcal{E}_{\infty}$. We denote by  $\mathcal{E}_{\infty}^{\mathrm{fin}}$ the collection of all equivalence relations on all initial segments $k=\{0,\ldots,k-1\}$ of $\mathbb{N}$. Notice that  $\mathcal{E}_{\infty}^{\mathrm{fin}}=\{r_n(A)\colon A\in \mathcal{E}_{\infty}, n\in\mathbb{N}\}$. If $s,t\in \mathcal{E}_{\infty}^{\mathrm{fin}}$, then  we write $s\leq t$ if $\mathrm{dom}(s)=\mathrm{dom}(t)$ and $s$ is coarser than $t$, \textit{i.e.},  if every $s$-class is a union of $t$-classes. 
Below, the implicit assumption is that  $s,t,\ldots$ denote elements of $\mathcal{E}_{\infty}^{\mathrm{fin}}$ and $A,B,\ldots$ elements of $\mathcal{E}_{\infty}$. We say that $A$ is an {\bf end-extension} of $s$ and write $s\sqsubseteq A$ if there is some $n\in\mathbb{N}$ so that $s=r_n(A)$.

It follows by  \cite{CS} that the triple $(\mathcal{E}_{\infty},\leq ,r)$ is a topological Ramsey space in the sense of \cite[Section 5.1]{Ramsey}; see \cite[Theorem 5.70]{Ramsey}. For every $s$ and $A$  let 
\[[s,A]:=\{B\in \mathcal{E}_{\infty} \colon s \sqsubseteq B  \text{  and }  B\leq A \}\]
be the associated {\bf Ellentuck cube}.
The {\bf Polish topology} on $\mathcal{E}_{\infty}$ is the topology generated by all Ellentuck cubes of the form $[s,\mathbb{N}_{\mathrm{ER}}]$, where $s$ ranges over  elements of  $\mathcal{E}_{\infty}^{\mathrm{fin}}$ and $\mathbb{N}_{\mathrm{ER}}$ is the $\leq$-maximal element $\{(n,n)\colon n\in \mathbb{N}\}$ of $\mathcal{E}_{\infty}$. The {\bf Ellentuck topology} on $\mathcal{E}_{\infty}$  is the topology generated by all Ellentuck cubes. Recall from \cite{Ramsey}:

\begin{definition}
A set  $\mathcal{X} \subseteq  \mathcal{E}_{\infty}$ is {\bf Ramsey} if for every $[s,A]\neq \emptyset$  there exists $B\in[s,A]$  such that either $[s,B]\subseteq \mathcal{X}$ or $[s,B]\subseteq \mathcal{X}^c$. It is {\bf Ramsey null} if for every $[s,A]\neq \emptyset$  there exists $B\in [s,A] $  such that  $[s,B] \subseteq \mathcal{X}^c$.
\end{definition}

The following is proved in \cite{CS}; see also \cite{Ramsey}.

\begin{theorem}[Carlson-Simpson]\label{T:CS}
For every $\mathcal{X}\subseteq  \mathcal{E}_{\infty}$ we have that :
\begin{enumerate}
\item $\mathcal{X}$ has the Baire property in the Ellentuck topology if and only if $\mathcal{X}$ is Ramsey;
\item $\mathcal{X}$ is meager in the Ellentuck topology if and only if $\mathcal{X}$ is Ramsey null.
\end{enumerate}
\end{theorem}

A subset $\mathcal{X}\subseteq \mathcal{E}_{\infty}$ is {\bf Borel} if it is Borel in the Polish topology of $\mathcal{E}_{\infty}$. Notice that this implies that $\mathcal{X}$ is Borel in the Ellentuck topology of $\mathcal{E}_{\infty}$ as well. It is not difficult to construct some  meager (in the Polish topology) $\mathcal{X}\subseteq \mathcal{E}_{\infty}$ which is Ramsey null and hence $\mathcal{X}^c$ is comeager but Ramsey null.   
\section{$E_1$ is in the spectrum of $\mathcal{E}_{\infty}$}

Consider the equivalence relations $E_1, E_{1}^{\mathrm{tail}}$ on $2^{\mathbb{N}\times \mathbb{N}}$ given by
\begin{align*}
x E_1 y &\iff \exists n_0\in\mathbb{N} \; \; \forall n, m \in\mathbb{N} \; \text{ so that }  x(n_0+n,m)= y(n_0+n,m), \\
x E_{1}^{\mathrm{tail}} y &\iff \exists n_x,n_y\in\mathbb{N} \; \; \forall n, m \in\mathbb{N} \; \text{ so that }  x(n_x+n,m)= y(n_y+n,m). 
\end{align*}
Consider also the $G_{\delta}$ subset $\mathcal{CS}\subseteq 2^{\mathbb{N}\times \mathbb{N}}$ consisting of all $x\in 2^{\mathbb{N}\times \mathbb{N}}$ such that the sets $A^x_0,A^x_1,A^x_2,\ldots$ given by $A^x_n:=\{m\in\mathbb{N}\colon x(n,m)=1\}$ are non-empty, form a partition of $\mathbb{N}$, and satisfy $\mathrm{min}(A_n^x) < \mathrm{min}(A_{n'}^x)$ whenever $n < n'$. Notice that $\mathcal{E}_{\infty}$ with its Polish topology is homeomorphic to $\mathcal{CS}$ under the map which sends $A\in\mathcal{E}_{\infty}$ to the unique $x\in\mathcal{CS}$ with $A^x_n:= \{m \in \mathbb{N} \colon m~A~\mu_n(A)\}$. That is, to the unique  $x\in\mathcal{CS}$ so that the associated sets $A^x_0,A^x_1,A^x_2,\ldots$ form a list of all equivalence classes of $A$ whose order in the list coincides with the order of their min elements. In what follows we will freely switch between the notation $\mathcal{CS}$ and  $\mathcal{E}_{\infty}$.

\begin{lemma}\label{L:E1_1}
The equivalence relations $(2^{\mathbb{N}\times\mathbb{N}},E_{1})$ and $(\mathcal{CS},E_{1}^{\mathrm{tail}})$ are Borel bireducible.
\end{lemma}
\begin{proof}
Clearly the inclusion map $i\colon \mathcal{CS} \to 2^{\mathbb{N}\times\mathbb{N}}$ induces a continuous reduction
\[(\mathcal{CS},E_{1}^{\mathrm{tail}})\leq_B (2^{\mathbb{N}\times\mathbb{N}},E_{1}^{\mathrm{tail}})\] 
By \cite[Theorem 8.1]{DJK} we have that $(2^{\mathbb{N}\times \mathbb{N}},E_{1}^{\mathrm{tail}})$ is hypersmooth. In particular,  $(2^{\mathbb{N}\times \mathbb{N}},E_{1}^{\mathrm{tail}})\leq_B (2^{\mathbb{N}\times \mathbb{N}},E_1)$ in a strong way; see \cite[Proposition 1.3]{KL}. Composing these two  Borel reductions we have that $(\mathcal{CS},E_{1}^{\mathrm{tail}})\leq_B (2^{\mathbb{N}\times\mathbb{N}},E_{1})$.

For the other direction, we define a Borel reduction  $f\colon 2^{\mathbb{N}\times\mathbb{N}} \to \mathcal{CS}$ as follows. First let \[\mathbb{N} = N_{-1}\bigsqcup N_{0}\bigsqcup N_{1}\bigsqcup N_{2}\bigsqcup\cdots \]
be  a partition of $\mathbb{N}$ into infinitely many infinite sets and for every $k\geq -1$ fix some bijection $t_k\colon \mathbb{N} \to N_k$. Let now $x\in 2^{\mathbb{N}\times\mathbb{N}}$ and consider the sets $B_0,B_1,B_2,\ldots \subseteq \mathbb{N}$ given by $B_n:=\{m\in\mathbb{N}\colon x(n,m)=1\}$. For every $n\in\mathbb{N}$ we define the sets $A_{2n}, A_{2n+1}$ by 
\[A_{2n}:= t_{n}(B_n)\cup \{t_{-1}(2n)\} \quad \text{ and } \quad A_{2n+1}:= (N_{n} \setminus t_{n}(B_n))\cup \{t_{-1}(2n+1)\}\]
Notice that the sets $A_0,A_1,A_2,A_3,\ldots$  are non-empty and form a partition of $\mathbb{N}$. Moreover, by making sure that $t_{-1}\colon \mathbb{N}\to N_{-1}$ is increasing and that both $t_{-1}(2n), t_{-1}(2n+1)$ appear before the least element of $N_n$ we have that $A_0,A_1,\ldots$ is of the form $A^y_0,A^y_1,\ldots$ for some $y\in\mathcal{CS}$. Set $f(x)=y$ for this $y$.

It remains to check that $f$ is a reduction from $E_1$ to $E_{1}^{\mathrm{tail}}$. Let $x,x'\in 2^{\mathbb{N}\times\mathbb{N}}$. 
It is easy to see that if $x \; E_1 \; x'$ implies   $f(x) \; E_1 \; f(x')$, which implies  $f(x) \; E_{1}^{\mathrm{tail}} \; f(x')$.
The fact that $f(x) \; E_{1}^{\mathrm{tail}} \; f(x')$ implies $x \; E_1 \; x'$ follows from the injectivity of $t_{-1}$.

\end{proof}

\begin{lemma}\label{L:E1_2}
Let $D\in\mathcal{E}_{\infty}$. Then $E_{1}^{\mathrm{tail}}$ on $\mathcal{E}_{\infty}$ is Borel isomorphic to $E_{1}^{\mathrm{tail}}\upharpoonright [\emptyset,D]$.
\end{lemma}
\begin{proof}
Let $D_0,D_1,D_2,\ldots$ be the list of all equivalence classes of $D$ in increasing order. Let $f\colon \mathcal{E}_{\infty} \to [\emptyset,D]$ be defined as follows: if $A\in \mathcal{E}_{\infty}$ and  $A_0,A_1,A_2,\ldots$
 is the list of all equivalence classes of $A$ in increasing order, then let $f(A)$ be the equivalence relation $B\leq D$ whose $n$-th equivalence class $B_n$ in the usual increasing order is the set $\bigcup_{k\in A_n}D_k$.

Then $f$ is a Borel isomorphism which respects $E_{1}^{\mathrm{tail}}$.
\end{proof}

Proposition \ref{P:1} is a straightforward consequence of Lemma \ref{L:E1_1} and Lemma \ref{L:E1_2}.

\section{Borel maps are tracial on a Ramsey conull set}\label{S:tracial}

Let $A,B\in \mathcal{E}_{\infty}$ and $n\in\mathbb{N}$. The {\bf trace of $A$ in $B$ at $n$}, denoted $\mathrm{tr}(A,B,n)$, is the restriction of the equivalence relation $A$ to the finite set $\{0,1,\ldots,\mu_n(B)-1\}$. Notice that  $\mathrm{tr}(A,B,n)$ may not be equal to $r_m(A)$ for any $m\in\mathbb{N}$.

\begin{definition}\label{D:Trace}
Let $f\colon \mathcal{E}_{\infty}\to \mathcal{E}_{\infty}$ be a map and  $\mathcal{X}\subseteq \mathcal{E}_{\infty}$. We say that $f$ is {\bf tracial on $\mathcal{X}$} if for every $B\in \mathcal{X}$ there is $n_0\in\mathbb{N}$ so that for all $n\geq n_0$ and $A,A'\leq B$ for which there exist $k\in\mathbb{N}$ and $s\in\mathcal{E}_{\infty}^{\mathrm{fin}}$ with $s=r_k(A)=r_k(A')$ and $\mathrm{dom}(s)=\{0,\ldots,\mu_n(B)-1\}$, we have
\begin{align}\label{Eq:Tracial}
 \mathrm{tr} \big(f(A),B,n\big) = \mathrm{tr}\big( f(A'),B,n\big).
\end{align}
\end{definition}

The main result of this section is the following lemma.

\begin{lemma}[Traciality]\label{L:Tracial}
Let  $f\colon \mathcal{E}_{\infty}\to \mathcal{E}_{\infty}$ be a Borel map. Then there exists a Ramsey conull set  $\mathcal{C}\subseteq  \mathcal{E}_{\infty}$ so that   $f$ is tracial on  $\mathcal{C}$.
\end{lemma}

The proof of Lemma \ref{L:Tracial} does not rely on any special properties of the Carlson-Simpson space $\mathcal{E}_{\infty}$. Indeed, for any topological Ramsey space $\mathcal{R}$ and any notion of ``trace" on $\mathcal{R}$, one can state and prove a lemma similar to Lemma \ref{L:Tracial}. Since there are other topological Ramsey spaces admitting useful notions of a trace we will restate Definition \ref{D:Trace} and prove Lemma \ref{L:Tracial} for any trace on the arbitrary topological Ramsey space $\mathcal{R}$.

Let $(\mathcal{R},\leq, r)$ be a topological Ramsey space in the sense of \cite[Section 5.1]{Ramsey}, and let $\mathcal{AR}$ be its set of finite approximations. A {\bf trace on $\mathcal{R}$} is any family $\mathcal{T}= \{\mathcal{T}_t\colon t\in\mathcal{AR}\}$ of finite Borel partitions $\mathcal{T}_t$  of $\mathcal{R}$, one for every  $t\in\mathcal{AR}$. Let $\mathcal{T}$ be a trace on $\mathcal{R}$. For all $A,B\in\mathcal{R}$ and $n\in\mathbb{N}$, we write  $\mathrm{tr}_{\mathcal{T}} \big(A,B,n\big)$ to denote the piece in the partition $\mathcal{T}_{r_n(B)}$ that contains $A$. 
\begin{example}
Let $\mathcal{T}$ be the trace on $\mathcal{E}_{\infty}$  that is defined as follows: for all $t\in\mathcal{E}_{\infty}^{\mathrm{fin}}$, we have that $A,A'$ are in the same pieces of the partition $\mathcal{T}_t$ if  $A\upharpoonright \mathrm{dom}(t)=A'\upharpoonright \mathrm{dom}(t)$.  It follows that  $\mathrm{tr}_{\mathcal{T}} \big(A,B,n\big)= \mathrm{tr}_{\mathcal{T}} \big(A',B,n\big)$ if and only if $\mathrm{tr} \big(A,B,n\big)= \mathrm{tr} \big(A',B,n\big)$, where the latter is defined at the beginning of this section.
\end{example}

\begin{example}\label{Example:Tree}
Let $2^{<\mathbb{N}}$ be the full binary tree. Recall that a strong subtree $A$ of $2^{<\mathbb{N}}$ is any subset of $2^{<\mathbb{N}}$ for which there are is a sequence $k_0<k_1<\ldots$ in $\mathbb{N}$ so that $A$ can be partitioned into sets $A(\ell) \subseteq 2^{k_{\ell}}$ satisfying
\begin{enumerate}
    \item $A(0)$ is a singleton (the root of $A$);
    \item for every $a\in A(\ell)$ and $i \in \{0, 1\}$, there is exactly one $b \in A(\ell+1)$ extending $a^\frown i$.
\end{enumerate}
We denote by $\mathcal{S}_{\infty}(2^{<\mathbb{N}})$ the Milliken space of all strong subtrees of $2^{<\mathbb{N}}$; see \cite[Section 6.1]{Ramsey}.
For every $A\in \mathcal{S}_{\infty}(2^{<\mathbb{N}})$, let $\langle A\rangle$ be the collection of all $b\in 2^{<\mathbb{N}}$ for which there is some $a\in A$ extending $b$. In particular, we always have that $\emptyset\in \langle A\rangle$.  Let $\mathcal{T}$ be the trace on $\mathcal{S}_{\infty}(2^{<\mathbb{N}})$   that is defined as follows: for all $t\in \mathcal{S}_{<\infty}(2^{<\mathbb{N}})$ with $t=T(0)\cup\cdots \cup T(\ell)$ for some $T\in \mathcal{S}_{\infty}(2^{<\mathbb{N}})$,
$A,A'\in \mathcal{S}_{\infty}(2^{<\mathbb{N}})$ are in the same piece of the partition $\mathcal{T}_t$ if  \[\langle A\rangle\cap T(\ell)= \langle A'\rangle\cap T(\ell).\]
\end{example}

For the rest of this section we fix some topological Ramsey space $\mathcal{R}$ together with some trace $\mathcal{T}$ on $\mathcal{R}$. Next we state analogues of Definition \ref{D:Trace} and Lemma \ref{L:Tracial} for the pair $\mathcal{R},\mathcal{T}$.

\begin{definition}\label{D:TraceAbstract}
Let $f\colon \mathcal{R}\to \mathcal{R}$ be a map and $\mathcal{X}\subseteq \mathcal{R}$. We say that $f$ is {\bf tracial on $\mathcal{X}$} if for every $B\in \mathcal{X}$ there is  $n_0\in\mathbb{N}$ so that for every $n\geq n_0$ and all  $A,A'\leq B$ for which there are $k\in\mathbb{N}$ and $s\in\mathcal{AR}$ with $s=r_k(A)=r_k(A')$ and $\mathrm{depth}_B(s)=n$, we have that
\begin{align}\label{Eq:AbstractTracial}
  \mathrm{tr}_{\mathcal{T}} \big(f(A),B,n\big) = \mathrm{tr}_{\mathcal{T}}\big( f(A'),B,n\big).
\end{align}
\end{definition}

Notice that for the topological Ramsey space $\mathcal{E}_{\infty}$, if $A\in \mathcal{E}_{\infty}$ and $s\in \mathcal{E}_{\infty}^{\mathrm{fin}}$, then we have that  $\mathrm{depth}_B(s)=n$ if and only if  $\mathrm{dom}(s)=\{0,\ldots,\mu_n(B)-1\}$. As a consequence, Definition \ref{D:Trace} is indeed a special case of Definition \ref{D:TraceAbstract}.

\begin{proposition}\label{Prop:Tracial}
Let  $f\colon \mathcal{R}\to \mathcal{R}$ be a Borel map. Then there exists a Ramsey conull set  $\mathcal{C}\subseteq  \mathcal{R}$ so that   $f$ is tracial on  $\mathcal{C}$.
\end{proposition}
\begin{proof}
Let $D\in\mathcal{R}$ and $\tau\in \mathcal{AR}$ with  $[\tau, D]\neq \emptyset$. Without loss of generality, we may assume that $D\in [\tau, D]$. Set  $n_0:= \mathrm{depth}_{D}(\tau)$. We will find $D'\in [\tau, D]$ so that for every $n\geq n_0$ and all $A,A'\leq D'$ for which there are some $k\in\mathbb{N}$ and $s\in\mathcal{AR}$ with $s=r_k(A)=r_k(A')$ and $\mathrm{depth}_D(s)=n$, we have that $\mathrm{tr}_{\mathcal{T}} \big(f(A),D',n\big) = \mathrm{tr}_{\mathcal{T}}\big( f(A'),D',n\big)$.

To define $D'$ we  first build a sequence $(D_n\colon n\geq n_0)$ with
\begin{align}\label{Eq1}
D_{n+1}\in [ r_n(D_{n}), D_{n} ],   \text{ where } D_{n_0}=D,
\end{align}
so that for all  $A,A'\leq D_{n+1}$ for which there are some $k\in\mathbb{N}$ and $s\in\mathcal{AR}$ with $s=r_k(A)=r_k(A')$ and $\mathrm{depth}_{D_{n}}(s)=n$, we have that  
\begin{align}\label{Eq:2}
\mathrm{tr}_{\mathcal{T}} \big(f(A),D_n,n\big) = \mathrm{tr}_{\mathcal{T}}\big( f(A'),D_n,n\big).
\end{align}

Assuming we have defined $(D_n)_{n\geq n_0}$ as above, we let $D'$ be the  limit of $D_n$ in the Polish topology of $\mathcal{R}$. This limit exists by (\ref{Eq1}) above and  Axioms A.2 in \cite[Section 5.1]{Ramsey}. It
is the unique $D'$ given by  $r_n(D') = r_n(D_n) \text{ for all } n\geq n_0$.
Clearly  $D'\in [r_n(D_n), D_n]$   for all $n\geq n_0$. In particular, the case $n=n_0$ implies  $D'\in[\tau, D]$, while the general case $n\geq n_0$ implies 
that for all $n\geq n_0$  and all $A,A'\leq D'$ for which there are some $k\in\mathbb{N}$ and $s\in\mathcal{AR}$ with $s=r_k(A)=r_k(A')$ and $\mathrm{depth}_D(s)=n$, we have that $\mathrm{tr} \big(f(A),D',n\big) = \mathrm{tr}\big( f(A'),D',n\big)$. Indeed, let $n \geq n_0$ and $A,A'\leq D'$ so that there are some $k\in\mathbb{N}$ and $s\in\mathcal{AR}$ with $s=r_k(A)=r_k(A')$ and $\mathrm{depth}_D(s)=n$.
Since $r_n(D')=r_n(D_n)$ and $A,A'\leq D'\leq D_{n+1}$, by (\ref{Eq:2}) we have that
\[ \mathrm{tr}_{\mathcal{T}}(f(A),D_n,n) = \mathrm{tr}_{\mathcal{T}}(f(A'),D_n,n).
\]
Once again, since $r_n(D')=r_n(D_n)$, we have that $\mathrm{tr}_{\mathcal{T}}(f(A),D',n) = \mathrm{tr}_{\mathcal{T}}(f(A'),D',n).$

We  turn  now to the construction of $(D_n)_{n\geq n_0}$. For the base case, set $D_{n_0}:=D\in [\tau,D]$.  For the inductive step, assume that $D_n$ has been defined. We define $D_{n+1}$ as follows.

Set $r=r_n(D_n)$ and let $s_1,\ldots,s_N$ be an enumeration of all $s\in \mathcal{AR}$ with $\mathrm{depth}_{D_n}(s)=n$. Notice that for all such $s$ we have that $s\leq r$ and therefore, by Axiom A.2.(1) in \cite[Section 5.1]{Ramsey}, the collection of such $s$ is indeed finite.

We will define $D_{n+1}$ to be $C_N$, where $C_N\leq \ldots\leq C_0$ will be defined  by a second induction as follows.  Set $C_0:=D_n$.  For the inductive step assume  that $C_{i-1}$ has been defined for some $i\in \{1,\ldots,N\}$ with $C_{i-1}\in [r, D_n]$.  By Axiom A.3.(1) in \cite[Section 5.1]{Ramsey}, we have that $[s_i, C_{i-1}]\neq \emptyset$. Let $\gamma\colon [s_i, C_{i-1}]\to \mathcal{T}_r$ be given by
\[\gamma(A):=\mathrm{tr}_{\mathcal{T}}(f(A),D_n,n).
\]
Since $\gamma$ is Borel and $\mathcal{T}_r$ is finite, by  \cite[Corollary 5.12]{Ramsey} there is some $C \in  [s_i, C_{i-1}]$ so that $\gamma$ is constant on $[s_i, C]$. Since  $[s_i, C]\neq \emptyset$, there is some $C'\in [r, C_{i-1}]$ so that
\[\emptyset \neq [s_i, C'] \subseteq [s_i, C].\]
This follows by Axiom A.3.(2) in \cite[Section 5.1]{Ramsey}. We finish the inductive step by setting $C_{i}:=C'$ and noticing that $\gamma$ is constant on $[s_i, C_i]$. 

Set now $D_{n+1}:= C_N$. We have (\ref{Eq1}) since $C_N\leq C_0=D_n$ and $r= r_n(D_n)$. Moreover, consider any 
$A,A'\leq D_{n+1}$ for which there are some $k\in\mathbb{N}$ and $s\in\mathcal{AR}$ with $s=r_k(A)=r_k(A')$ and $\mathrm{depth}_{D_{n}}(s)=n$.
Then $s=s_i$ for some $i$. But then $A,A'\in [s_i, D_{n+1}]\subseteq  [s_i, C_{i}]$, and since  $\gamma$ is constant on $[s_i, C_{i}]$, we have
\[ \mathrm{tr}_{\mathcal{T}}(f(A),D_n,n) = \mathrm{tr}_{\mathcal{T}}(f(A'),D_n,n).\]
So (\ref{Eq:2}) holds as well.
\end{proof}

\section{A combinatorial lemma}

For any $n \in \mathbb{N} \setminus \{0\}$, let $\mathcal{Q}(n)$ denote the set of partitions of $n = \{0, 1, \ldots, n - 1\}$. We treat members of $\mathcal{Q}(n)$ as equivalence relations on $\{0, 1, \ldots, n - 1\}$ and let $\mathcal{Q}(n)$ inherit the relation $\leq$ from $\mathcal{E}_{\infty}^{\mathrm{fin}}$. For $m \leq n$, let
\begin{align*}
    \mathcal{Q}^m(n) := \{s \in \mathcal{Q}(n) \colon i < j < m \implies \neg \; i \; s \; j\}.
\end{align*}
In words, $\mathcal{Q}^m(n)$ denotes the set of partitions of $n$ such that $0, 1, \ldots, m - 1$ belong to different equivalence classes.
The following is the main result of this section. 

\begin{lemma}\label{L:CombGen}
For every $m,M \in \mathbb{N}$ there exists $M'\geq m,M$ such that the following holds: for every map $e \colon \mathcal{Q}(M') \to \mathcal{Q}(M')$, there exists $s \in \mathcal{Q}^m(M')$ such that for all $t \leq s$, either
\begin{align*}
    e(t) \leq t \; \mathrm{or} \; e(t) \nleq s.
\end{align*}
\end{lemma}

For the proof of Lemma \ref{L:CombGen} we will rely on a probabilistic argument: we will show that the ``random'' $s \in \mathcal{Q}^m(M')$ satisfies the desired property with non-zero probability. To avoid dealing with expressions that involve Stirling numbers of the second kind in our counting arguments, we will endow $\mathcal{Q}^m(M')$ with the uniform measure supported only on ``equipartitions.'' This allows us to state and prove the following more specific result that implies Lemma \ref{L:CombGen}.

For $k, N \in \mathbb{N} \setminus \{0\}$, let $\mathcal{Q}_k(kN)$ denote the set of partitions of $kN$ into $k$ classes each having $N$ elements. For $m \leq k$, we let $\mathcal{Q}_k^m(kN) := \mathcal{Q}^m(kN) \cap \mathcal{Q}_k(kN)$.

\begin{lemma}\label{L:Comb}
Consider any $m \in \mathbb{N}$ and $k \in \mathbb{N} \setminus \{0\}$ with $m \leq k$. Then there exists $N' \in \mathbb{N}$ such that for all $N \geq N'$, the following holds: for every map $e \colon \mathcal{Q}(kN) \to \mathcal{Q}(kN)$, there exists $s \in \mathcal{Q}_k^m(kN)$ such that for all $t \leq s$, either 
\begin{align}\label{CombCondition}
    e(t) \leq t \; \mathrm{or} \; e(t) \nleq s.
\end{align}
\end{lemma}

It is clear that Lemma \ref{L:CombGen} follows from Lemma \ref{L:Comb}.

\begin{proof}[Proof of Lemma \ref{L:Comb}.]
We first prove a claim to simplify our later calculations.

\begin{claim}
For all $k \in \mathbb{N} \setminus \{0\}$, there exists $M \in \mathbb{N}$ such that for all $N \geq M$, if $a_1, a_2, b_1, b_2 \in \mathbb{N}$ satisfy $0 \leq b_i \leq a_i \leq k$ and $a_i \neq 0$ for $i \in \{1, 2\}$, we have
\begin{align}\label{combine}
    \frac{(a_1 N - b_1)!}{(a_1 - b_1)!} \cdot \frac{(a_2 N - b_2)!}{(a_2 - b_2)!} \leq \frac{((a_1 + a_2) N - (b_1 + b_2))!}{((a_1 + a_2) - (b_1 + b_2))!}.
\end{align}
\end{claim}
\begin{proof}[Proof of Claim.]
Observe that (\ref{combine}) holds for all $N > 0$ if $a_1 = b_1$ and $a_2 = b_2$. Also note that there are finitely many choices of $a_1, a_2, b_1,$ and $b_2$ satisfying the hypotheses. So it suffices to show that
\begin{align*}
    R_{a_1, a_2, b_1, b_2} := \lim_{N \to \infty} \frac{(a_1 N - b_1)!}{(a_1 - b_1)!} \cdot \frac{(a_2 N - b_2)!}{(a_2 - b_2)!} \cdot \frac{((a_1 + a_2) - (b_1 + b_2))!}{((a_1 + a_2) N - (b_1 + b_2))!} < 1
\end{align*}
for all $a_1, a_2, b_1, b_2 \in \mathbb{N}$ such that $0 < a_1, a_2 \leq k$, $b_1 < a_1$, and $b_2 \leq a_2$.

The following bound is a consequence of Stirling's approximation formula: for any $a \in \mathbb{N}$ and $b \in \mathbb{N} \setminus \{0\}$ with $a \leq b$, we have
\begin{align}\label{entropyBound}
    \frac{1}{b + 1} 2^{b H(a / b)} \leq \binom{b}{a} \leq 2^{b H(a / b)},
\end{align}
where $H$ denotes the binary entropy function
\begin{align*}
    H \bigg( \frac{a}{b} \bigg) = - \frac{a}{b} \log_2 \frac{a}{b} - \bigg( 1 - \frac{a}{b} \bigg) \log_2 \bigg( 1 - \frac{a}{b} \bigg).
\end{align*}
Alternatively, this bound can be derived directly as in \cite[Example 11.1.3]{CoverThomas}. Using (\ref{entropyBound}), we obtain
\begin{align*}
    R_{a_1, a_2, b_1, b_2} & = \lim_{N \to \infty} \frac{(a_1 N - b_1)!}{(a_1 - b_1)!} \cdot \frac{(a_2 N - b_2)!}{(a_2 - b_2)!} \cdot \frac{((a_1 + a_2) - (b_1 + b_2))!}{((a_1 + a_2) N - (b_1 + b_2))!} \\
    & = \lim_{N \to \infty} \frac{\binom{(a_1 + a_2) - (b_1 + b_2)}{a_1 - b_1}}{\binom{(a_1 + a_2) N - (b_1 + b_2)}{a_1 N - b_1}} \\
    & \leq 2^{((a_1 + a_2) - (b_1 + b_2)) H(\frac{a_1 - b_1}{(a_1 + a_2) - (b_1 + b_2)})} \lim_{N \to \infty} \frac{(a_1 + a_2) N - (b_1 + b_2) + 1}{2^{((a_1 + a_2)N - (b_1 + b_2)) H(\frac{a_1 N - b_1}{(a_1 + a_2) N - (b_1 + b_2)})}}.
\end{align*}
Since $a_1, a_2 > 0$, note that $\lim_{N \to \infty} H(\frac{a_1 N - b_1}{(a_1 + a_2) N - (b_1 + b_2)}) = H(\frac{a_1}{a_1 + a_2}) > 0$. Because $\lim_{N \to \infty} \frac{c_1 N + d_1}{2^{c_2 N + d_2}} = 0$ in general for $c_2 > 0$, we have $R_{a_1, a_2, b_1, b_2} = 0$. In particular, $R_{a_1, a_2, b_1, b_2} < 1$, as desired.
\end{proof}

Fix $M$ as in the claim. Consider any $N \geq M$ and any map $e \colon \mathcal{Q}(kN) \to \mathcal{Q}(kN)$. Let $B_e$ consist of all pairs $(s, t) \in \mathcal{Q}_k^m(kN) \times \mathcal{Q}(kN)$ such that $t \leq s$ and (\ref{CombCondition}) fails. It suffices to find $N$ such that $|B_e| < |\mathcal{Q}_k^m(kN)|$ for every map $e \colon \mathcal{Q}(kN) \to \mathcal{Q}(kN)$.

Consider any $t \in \mathcal{Q}(kN)$, and let $h_t^e := t \wedge e(t)$, the partition of $kN$ such that
\begin{align*}
    i \; h_t^e \; j \iff i \; t \; j \text{ and } i \; e(t) \; j.
\end{align*}
Note that $t \leq h_t^e$ and $e(t) \leq h_t^e$. Fix any $t$ for which there exists $s \in \mathcal{Q}_k^m(kN)$ such that $(s, t) \in B_e$. Then observe that $t \lneq h_t^e$. Hence for such $t$, for any $s \in \mathcal{Q}_k^m(kN)$,
\begin{align}\label{bhrelation}
    (s, t) \in B_e \iff \big( t \leq s \text{ and } e(t) \leq s \big) \iff h_t^e \leq s.
\end{align}
Let
\begin{align*}
    r_t^e := \frac{|\{s \in \mathcal{Q}_k^m(kN) \colon h_t^e \leq s\}|}{|\{s \in \mathcal{Q}_k^m(kN) \colon t \leq s\}|}.
\end{align*}
We will bound $|B_e|$ by bounding $r_t^e$ for all such $t$ and using (\ref{bhrelation}).

Write $t = \{X_1, X_2, \ldots X_u\}$. Since $t$ is a coarsening of some $s \in \mathcal{Q}_k^m(kN)$, we can find positive $k_1, k_2, \ldots, k_u \in \mathbb{N}$ such that $|X_i| = k_i N$ for each $1 \leq i \leq u$. For each $i$, let $m_i := |X_i \cap m|$. Observe that $\sum_{1 \leq i \leq u} k_i = k$, $\sum_{1 \leq i \leq u} m_i = m$, and $m_i \leq k_i$ for each $i$. Now write $h_t^e = \{X_{i, j} \colon 1 \leq i \leq u, 1 \leq j \leq \ell_i\}$ such that for each $i$, $\bigcup_{1 \leq j \leq \ell_i} X_{i, j} = X_i$. Since we are assuming there exists $s \in \mathcal{Q}_k^m(kN)$ with $(s, t) \in B_e$, we know from (\ref{bhrelation}) that $h_t^e \leq s$ for some $s \in \mathcal{Q}_k^m(kN)$. So for each $1 \leq i \leq u$ and $1 \leq j \leq \ell_i$, we can write $|X_{i, j}| = k_{i, j} N$. Let $m_{i, j} = |X_{i, j} \cap m|$. Note that for each $i$, $\sum_{1 \leq j \leq \ell_i} k_{i, j} = k_i$ and $\sum_{1 \leq j \leq \ell_i} m_{i, j} = m_i$. Also, $m_{i, j} \leq k_{i, j}$ for all $i, j$ since $h_t^e \leq s$ for some $s \in \mathcal{Q}_k^m(kN)$.

We now bound $r_t^e$. Evaluating the denominator, we have
\begin{align*}
    |\{s \in \mathcal{Q}_k^m(kN) \colon t \leq s\}| & = \prod_{1 \leq i \leq u} \frac{(k_i N - m_i)!}{(N - 1)!^{m_i} \cdot N!^{k_i - m_i} \cdot (k_i - m_i)!} \\
    & = \frac{1}{(N - 1)!^m \cdot N!^{k - m}} \prod_{1 \leq i \leq u} \frac{(k_i N - m_i)!}{(k_i - m_i)!}.
\end{align*}
Similarly, evaluating the numerator of $r_t^e$ yields
\begin{align*}
    |\{s \in \mathcal{Q}_k^m(kN) \colon h_t^e \leq s\}| & = \frac{1}{(N - 1)!^m \cdot N!^{k - m}} \prod_{1 \leq i \leq u} \prod_{1 \leq j \leq \ell_i} \frac{(k_{i, j} N - m_{i, j})!}{(k_{i, j} - m_{i, j})!}.
\end{align*}
Since $N \geq M$, we can apply the claim to obtain
\begin{align*}
    |\{s \in \mathcal{Q}_k^m(kN) \colon h_t^e \leq s\}| & \leq \frac{1}{(N - 1)!^m \cdot N!^{k - m}} \prod_{1 \leq i \leq u} \frac{(k_{i, 1} N - m_{i, 1})!}{(k_{i, 1} - m_{i, 1})!} \cdot \frac{((k_i - k_{i, 1}) N - (m_i - m_{i, 1}))!}{((k_i - k_{i, 1}) - (m_i - m_{i, 1}))!}.
\end{align*}
So we have
\begin{align*}
    r_t^e & \leq \prod_{1 \leq i \leq u} \frac{\binom{k_i - m_i}{k_{i, 1} - m_{i, 1}}}{\binom{k_i N - m_i}{k_{i, 1} N - m_{i, 1}}}.
\end{align*}

Observe that if $k_{i, 1} = k_i$, then $m_{i, 1} = m_i$ and $\frac{\binom{k_i - m_i}{k_{i, 1} - m_{i, 1}}}{\binom{k_i N - m_i}{k_{i, 1} N - m_{i, 1}}} = 1$. Since $t \lneq h_t^e$, there must be some $1 \leq i \leq u$ for which $k_{i, 1} < k_i$. So
\begin{align}
    r_t^e & \leq \prod_{\substack{1 \leq i \leq u \\ k_{i, 1} < k_i}} \frac{\binom{k_i - m_i}{k_{i, 1} - m_{i, 1}}}{\binom{k_i N - m_i}{k_{i, 1} N - m_{i, 1}}} \nonumber \\
    & \leq \prod_{\substack{1 \leq i \leq u \\ k_{i, 1} < k_i}} \frac{2^{(k_i - m_i) H(\frac{k_{i, 1} - m_{i, 1}}{k_i - m_i})}}{\frac{1}{k_i N - m_i + 1} 2^{(k_i N - m_i) H(\frac{k_{i, 1} N - m_{i, 1}}{k_i N - m_i})}}, \label{combRatio}
\end{align}
where the last inequality follows by (\ref{entropyBound}). Note that (\ref{combRatio}) depends only on the values of $k_i, m_i, k_{i, 1}$, and $m_{i, 1}$ for each $1 \leq i \leq u$. Note that there are finitely many ways to choose these values, letting $u$ vary, such that the following necessary conditions hold:
\begin{enumerate}
    \item $0 < k_i \leq k$ for each $i$;
    \item $\sum_i k_i = k$;
    \item $0 \leq m_i \leq k_i$ for each $i$;
    \item $\sum_i m_i = m$;
    \item $0 < k_{i, 1} \leq k_i$ for each $i$;
    \item $0 \leq m_{i, 1} \leq \min{\{k_{i, 1}, m_i\}}$ for each $i$;
    \item $k_{i, 1} < k_i$ for some $i$.
\end{enumerate}
For each such choice of these values, the limit of (\ref{combRatio}) as $N \to \infty$ is 0. Thus, we can fix $N' \in \mathbb{N}$ such that for all $N \geq N'$ and each choice of values for $k_i, m_i, k_{i, 1},$ and $m_{i, 1}$ satisfying the above conditions, (\ref{combRatio}) is less than $\frac{1}{|\mathcal{Q}(k)|}$.

Now consider any $N \geq N'$ and any $e \colon \mathcal{Q}(kN) \to \mathcal{Q}(kN)$. Then
\begin{align*}
    \frac{|B_e|}{|\{(s, t) \in \mathcal{Q}_k^m(kN) \times \mathcal{Q}(kN) \colon t \leq s \}|} & = \frac{\sum_{t \in \mathcal{Q}(kN)} |\{s \in \mathcal{Q}_k^m(kN) \colon (s, t) \in B_e\}|}{\sum_{t \in \mathcal{Q}(kN)} |\{s \in \mathcal{Q}_k^m(kN) \colon t \leq s\}} \\
    & = \frac{\sum_{\substack{t \in \mathcal{Q}(kN) \\ \exists s (s, t) \in B_e}} |\{s \in \mathcal{Q}_k^m(kN) \colon h_t^e \leq s\}|}{\sum_{t \in \mathcal{Q}(kN)} |\{s \in \mathcal{Q}_k^m(kN) \colon t \leq s\}} \\
    & \leq \frac{\sum_{\substack{t \in \mathcal{Q}(kN) \\ \exists s (s, t) \in B_e}} |\{s \in \mathcal{Q}_k^m(kN) \colon h_t^e \leq s\}|}{\sum_{\substack{t \in \mathcal{Q}(kN) \\ \exists s (s, t) \in B_e}} |\{s \in \mathcal{Q}_k^m(kN) \colon t \leq s\}} \\
    & < \frac{1}{|\mathcal{Q}(k)|}
\end{align*}
by our choice of $N$. So
\begin{align*}
    |B_e| & < \frac{1}{|\mathcal{Q}(k)|} \cdot |\{(s, t) \in \mathcal{Q}_k^m(kN) \times \mathcal{Q}(kN) \colon t \leq s \}| \\
    & = \frac{1}{|\mathcal{Q}(k)|} \sum_{s \in \mathcal{Q}_k^m(kN)} |\{t \in \mathcal{Q}(kN) \colon t \leq s\}| \\
    & = \frac{1}{|\mathcal{Q}(k)|} \sum_{s \in \mathcal{Q}_k^m(kN)} |\mathcal{Q}(k)| \\
    & = |\mathcal{Q}_k^m(kN)|.
\end{align*}
Hence, there exists $s \in \mathcal{Q}_k^m(kN)$ satisfying (\ref{CombCondition}).
\end{proof}

\section{$\mathcal{E}_{\infty}$ canonizes CBERs to equality}

We now prove Theorem \ref{T:Main}. Recall the statement of the theorem:

\begin{theorem}\label{T:cubeEquality}
Let $E$ be a countable Borel equivalence relation on $\mathcal{E}_{\infty}$. Then there exists a Ramsey conull set $\mathcal{D} \subseteq \mathcal{E}_{\infty}$ such that for all $B \in \mathcal{D}$ and $A, A' \leq B$,
\begin{align*}
    A \; E \; A' \iff A = A'.
\end{align*}
\end{theorem}

The proof of Theorem \ref{T:cubeEquality} relies on the following lemma.

\begin{lemma}\label{L:main}
Let $f: \mathcal{E}_{\infty} \to \mathcal{E}_{\infty}$ be a Borel map. Then there exists a Ramsey conull set $\mathcal{D} \subseteq \mathcal{E}_{\infty}$ such that if $B \in \mathcal{D}$ and $A \leq B$, then
\begin{align*}
    f(A) \leq B \implies f(A) \leq A.
\end{align*}
\end{lemma}

\begin{proof}[Proof of Theorem \ref{T:cubeEquality}.]
Given a CBER $E$ on $\mathcal{E}_{\infty}$, we may fix Borel involutions $\{f_n \colon n \in \mathbb{N}\}$ on $\mathcal{E}_{\infty}$ such that $E = \bigcup_{n \in \mathbb{N}} f_n$. For each $f_n$, fix a Ramsey conull set $\mathcal{D}_n \subseteq \mathcal{E}_{\infty}$ as in Lemma \ref{L:main}. Then $\mathcal{D} := \bigcap_{n \in \mathbb{N}} \mathcal{D}_n$ is Ramsey conull. Suppose $B \in \mathcal{D}$ and $A, A' \leq B$ satisfy $A \; E \; A'$. Then we can find $n \in \mathbb{N}$ such that $f_n(A) = A'$ and $f_n(A') = A$. Since $B \in \mathcal{D}_n$, we have $A \leq A'$ and $A' \leq A$ by Lemma \ref{L:main}, hence $A = A'$.
\end{proof}

It remains to prove Lemma \ref{L:main}.

\begin{proof}[Proof of Lemma \ref{L:main}.]
Consider an arbitrary cube $[s, D] \subseteq \mathcal{E}_{\infty}$ such that $[s, D] \neq \emptyset$. We will find $B \in [s, D]$ such that if $A \leq B$, then $f(A) \leq B$ implies $f(A) \leq A$. Since
\begin{align*}
    \{C \in \mathcal{E}_{\infty} \colon \text{for all } A \leq C, f(A) \leq C \implies f(A) \leq A\}
\end{align*}
is downwards-closed, finding such $B$ is sufficient.

By Proposition \ref{Prop:Tracial}, we can fix a Ramsey conull set $\mathcal{C} \subseteq \mathcal{E}_{\infty}$ such that $f$ is tracial on $\mathcal{C}$. Fix any $B_0 \in \mathcal{C} \cap [s, D]$. Since $f$ is tracial on $\mathcal{C}$, we can find $n_0 \geq \mathrm{depth}_{B_0}(s)$ such that for all $n \geq n_0$, if $A, A' \leq B_0$ and there is $k \in \mathbb{N}$ such that $\mu_k(A) = \mu_k(A') = \mu_n(B_0)$, then $r_k(A) = r_k(A')$ implies $\mathrm{tr}(f(A), B_0, n) = \mathrm{tr}(f(A'), B_0, n)$. Recall from Section \ref{S:tracial} that $\mathrm{depth}_{B_0}(s)$ is the unique $m$ such that $\mathrm{dom}(s) = \{0, \ldots, \mu_m(B_0) - 1 \}$. Since $n_0 \geq \mathrm{depth}_{B_0}(s)$ and $s \sqsubseteq B_0$, observe that $s \sqsubseteq r_{n_0}(B_0)$.

We will construct a sequence $(B_\ell \colon \ell \in \mathbb{N})$ such that the following properties hold for all $\ell \in \mathbb{N}$:
\begin{enumerate}[($1_\ell$)]
    \item $B_{\ell + 1} \in [r_{n_0 + \ell}(B_\ell), B_\ell]$;
    \item if $A \leq B_{\ell + 1}$ and there is $k$ such that $\mu_k(A) = \mu_{n_0 + \ell + 1}(B_{\ell + 1})$, then either
    \begin{align*}
        & \mathrm{tr}(f(A), B_{\ell + 1}, n_0 + \ell + 1) \leq \mathrm{tr}(A, B_{\ell + 1}, n_0 + \ell + 1) \text{ or } \\
        & \mathrm{tr}(f(A), B_{\ell + 1}, n_0 + \ell + 1) \nleq r_{n_0 + \ell + 1}(B_{\ell + 1}).
    \end{align*}
\end{enumerate}

Suppose we have constructed such a sequence. Let $B$ be the limit of $(B_\ell \colon \ell \in \mathbb{N})$ in the Polish topology on $\mathcal{E}_{\infty}$. Note that condition $(1_\ell)$ implies that such $B$ exists. Observe that $B \in [r_{n_0 + \ell}(B_\ell), B_\ell]$ for all $\ell \in \mathbb{N}$, using $(1_\ell)$ and Axiom A.2.(2) in \cite[Section 5.1]{Ramsey}. In particular, $B \in [r_{n_0}(B_0), B_0]$, so $B \in [s, D]$. We claim that $B$ satisfies the desired condition:

\begin{claim}
If $A \leq B$ and $f(A) \leq B$, then $f(A) \leq A$.
\end{claim}

\begin{proof}[Proof of Claim.]
Consider any $A \leq B$ with $f(A) \leq B$. We can find $\ell_0 < \ell_1 < \ldots$ in $\mathbb{N}$ such that for each $i$, there exists $m_i \in \mathbb{N}$ such that $\mu_{m_i}(A) = \mu_{n_0 + \ell_i + 1}(B) = \mu_{n_0 + \ell_i + 1}(B_{\ell_i + 1})$. For each $i \in \mathbb{N}$, condition $(2_{\ell_i})$ implies
\begin{align*}
    & \mathrm{tr}(f(A), B_{\ell_i + 1}, n_0 + \ell_i + 1) \leq \mathrm{tr}(A, B_{\ell_i + 1}, n_0 + \ell_i + 1) \text{ or } \\
    & \mathrm{tr}(f(A), B_{\ell_i + 1}, n_0 + \ell_i + 1) \nleq r_{n_0 + \ell_i + 1}(B_{\ell_i + 1}).
\end{align*}
Since $r_{n_0 + \ell_i + 1}(B) = r_{n_0 + \ell_i + 1}(B_{\ell_i + 1})$, we have
\begin{align*}
    \mathrm{tr}(f(A), B, n_0 + \ell_i + 1) \leq \mathrm{tr}(A, B, n_0 + \ell_i + 1) \text{ or } \mathrm{tr}(f(A), B, n_0 + \ell_i + 1) \nleq r_{n_0 + \ell_i + 1}(B).
\end{align*}
We are assuming $f(A) \leq B$, so we must have 
\begin{align*}
    \mathrm{tr}(f(A), B, n_0 + \ell_i + 1) = f(A) \upharpoonright \mu_{n_0 + \ell_i + 1}(B) \leq r_{n_0 + \ell_i + 1}(B).
\end{align*}
So for each $i$, we have 
\begin{align*}
    f(A) \upharpoonright \mu_{n_0 + \ell_i + 1}(B) = \mathrm{tr}(f(A), B, n_0 + \ell_i + 1) \leq \mathrm{tr}(A, B, n_0 + \ell_i + 1) = A \upharpoonright \mu_{n_0 + \ell_i + 1}(B).
\end{align*} 
Since this relation holds for every $i \in \mathbb{N}$ and $\ell_i \to \infty$ as $i \to \infty$, we conclude that $f(A) \leq A$.
\end{proof}

It remains to construct $(B_\ell \colon \ell \in \mathbb{N})$. Suppose we have constructed $B_\ell$ for some $\ell \in \mathbb{N}$. Fix $M'$ satisfying Lemma \ref{L:CombGen} in the case $m = M = n_0 + \ell + 1$. We define auxiliary maps $h \colon \mathcal{Q}(M') \to [\emptyset, B_\ell]$ and $g \colon \mathcal{E}_{\infty} \to \mathcal{Q}(M')$. Given any $t \in \mathcal{Q}(M')$, let $h(t) \in [\emptyset, B_\ell]$ be defined as follows: for $i, j \in \mathbb{N}$ with $i \; B_\ell \; \mu_{k_i}(B_\ell)$ and $j \; B_\ell \; \mu_{k_j}(B_\ell)$,
\begin{align*}
    i \; h(t) \; j \iff k_i = k_j \text{ or } \big(k_i, k_j < M' \text{ and } k_i \; t \; k_j \big).
\end{align*}
Intuitively, $h(t)$ is the coarsening of $B_\ell$ induced by $t$. For any $C \in \mathcal{E}_{\infty}$ define $g(C)$ as follows. If $C \upharpoonright \mu_{M'}(B_\ell) \leq r_{M'}(B_\ell)$, let $g(C) \in \mathcal{Q}(M')$ be defined by
\begin{align*}
    i \; g(C) \; j \iff \mu_i(B_\ell) \; C \; \mu_j(B_\ell)
\end{align*}
for all $i, j < M'$. Otherwise, if $C \upharpoonright \mu_{M'}(B_\ell) \nleq r_{M'}(B_\ell)$, let $i \; g(C) \; j$ if and only if $i = j$. Observe that $h(g(C)) \upharpoonright \mu_{M'}(B_\ell) = C \upharpoonright \mu_{M'}(B_\ell)$ whenever $C \upharpoonright \mu_{M'}(B_\ell) \leq r_{M'}(B_\ell)$. In particular, this equality holds if $C \leq B_\ell$.

Now define $e \colon \mathcal{Q}(M') \to \mathcal{Q}(M')$ by $e(t) := g(f(h(t)))$. By Lemma \ref{L:CombGen}, we can find $s' \in \mathcal{Q}^{n_0 + \ell + 1}(M')$ such that for all $t \leq s'$, either $e(t) \leq t$ or $e(t) \nleq s'$. Let $B_{\ell + 1} := h(s')$. We will show that $B_{\ell + 1}$ satisfies conditions $(1_\ell)$ and $(2_\ell)$. By definition of $h$, we know $B_{\ell + 1} \leq B_\ell$. Since $s' \in \mathcal{Q}^{n_0 + \ell + 1}(M')$, note that if $i < j < n_0 + \ell + 1$, then $\neg \; i \; s' \; j$, hence $\neg \; \mu_i(B_\ell) \; B_{\ell + 1} \; \mu_j(B_\ell)$. Then $\mu_i(B_\ell) = \mu_i(B_{\ell + 1})$ for all $i < n_0 + \ell + 1$, and so we have $r_{n_0 + \ell}(B_\ell) = r_{n_0 + \ell}(B_{\ell + 1})$. Hence, $(1_\ell)$ is satisfied.

To verify $(2_\ell)$, suppose we have $A \leq B_{\ell + 1}$ such that for some $k$, $\mu_k(A) = \mu_{n_0 + \ell + 1}(B_{\ell + 1})$. Furthermore, suppose $\mathrm{tr}(f(A), B_{\ell + 1}, n_0 + \ell + 1) \leq r_{n_0 + \ell + 1}(B_{\ell + 1})$. We will show that $\mathrm{tr}(f(A), B_{\ell + 1}, n_0 + \ell + 1) \leq \mathrm{tr}(A, B_{\ell + 1}, n_0 + \ell + 1)$. Since $B_{\ell + 1} \leq B_\ell \leq B_0$, there is some $n \geq n_0 + \ell + 1$ such that $\mu_{n_0 + \ell + 1}(B_{\ell + 1}) = \mu_n(B_0)$. Then observe that 
\begin{align*}
    \mu_k(A) = \mu_n(B_0) = \mu_{n_0 + \ell + 1}(B_{\ell + 1}) = \mu_{M'}(B_\ell).
\end{align*}
Since $A \leq B_{\ell + 1} \leq B_\ell$, we know $h(g(A)) \upharpoonright \mu_{M'}(B_\ell) = A \upharpoonright \mu_{M'}(B_\ell)$ and so $r_k(A) = r_k(h(g(A)))$. Since $\mu_n(B_0) = \mu_k(A) = \mu_k(h(g(A)))$ and $r_k(A) = r_k(h(g(A)))$, the traciality condition implies $f(A) \upharpoonright \mu_n(B_0) = f(h(g(A))) \upharpoonright \mu_n(B_0)$. We can rewrite this equality as $f(A) \upharpoonright \mu_{M'}(B_\ell) = f(h(g(A))) \upharpoonright \mu_{M'}(B_\ell)$. Then $g(f(A)) = g(f(h(g(A)))) = e(g(A))$, where the first equality is by definition of $g$ and the second is by definition of $e$. By our assumption that $\mathrm{tr}(f(A), B_{\ell + 1}, n_0 + \ell + 1) \leq r_{n_0 + \ell + 1}(B_{\ell + 1})$, note that $e(g(A)) = g(f(A)) \leq s'$. We have $g(A) \leq s'$ since $A \leq B_{\ell + 1}$, so $g(f(A)) = e(g(A)) \leq g(A)$ by choice of $s'$. Since $A \leq B_{\ell + 1}$ and $f(A) \upharpoonright \mu_{M'}(B_\ell) \leq r_{n_0 + \ell + 1}(B_{\ell + 1}) \leq r_{M'}(B_\ell)$, we conclude
\begin{align*}
    \mathrm{tr}(f(A), B_{\ell + 1}, n_0 + \ell + 1) & = f(A) \upharpoonright \mu_{n_0 + \ell + 1}(B_{\ell + 1}) \\
    & = f(A) \upharpoonright \mu_{M'}(B_\ell) \\
    & \leq A \upharpoonright \mu_{M'}(B_\ell)
\end{align*}
from the definition of $g$ and the fact that $g(f(A)) \leq g(A)$. Thus, $(2_\ell)$ holds.
\end{proof}

\section{$\mathcal{S}_{\infty}(2^{<\mathbb{N}})$ canonizes CBERs to equality}\label{S:Last} 

In this section we sketch how our arguments adapt  to prove  \cite[Theorem 8.1]{KSZ} for the special case of CBERs. Recall the definitions and notation from Example \ref{Example:Tree}.

\begin{theorem}[Kanovei-Sabok-Zapletal \cite{KSZ}] \label{T:trees}
Let $E$ be a countable Borel equivalence relation on $\mathcal{S}_{\infty}(2^{<\mathbb{N}})$. Then there is a Ramsey conull  $\mathcal{D}\subseteq \mathcal{S}_{\infty}(2^{<\mathbb{N}})$ so that for all $B\in \mathcal{D}$ and every $A,A'\leq B$ we have 
\[A \; E \; A' \iff A=A'.\]
\end{theorem}

As in the case of Theorem \ref{T:cubeEquality}, Theorem \ref{T:trees} directly follows from the following lemma.

\begin{lemma}\label{L:mainTree}
Let $f: \mathcal{S}_{\infty}(2^{<\mathbb{N}}) \to \mathcal{S}_{\infty}(2^{<\mathbb{N}})$ be a Borel map. Then there exists a Ramsey conull set $\mathcal{D} \subseteq \mathcal{S}_{\infty}(2^{<\mathbb{N}})$ such that if $B \in \mathcal{D}$ and $A \leq B$, then
\begin{align*}
    f(A) \leq B \implies f(A) \leq A.
\end{align*}
\end{lemma}

The proof of Lemma \ref{L:mainTree} is along the lines of the proof of Lemma \ref{L:main}. 
In particular, one starts with an arbitrary non-empty cube $[s,D]\subseteq \mathcal{S}_{\infty}(2^{<\mathbb{N}})$ and, using Proposition \ref{Prop:Tracial} for the trace $\mathcal{T}$ on $\mathcal{S}_{\infty}(2^{<\mathbb{N}})$   that was defined in Example \ref{Example:Tree}, one passes to some $B_0 \in [s,D]$ for which there is some  $n_0\in \mathbb{N}$ so that for all $n\geq n_0$ and all $A,A'\leq B_0$ with $A\cap B_0(n-1) = A'\cap B_0(n-1) \neq \emptyset$ we have
\begin{align}
\langle f(A)\rangle\cap B_0(n-1)=\langle f(A')\rangle\cap B_0(n-1). 
\end{align}

Then, one repeats the construction as in the proof of Lemma \ref{L:main}, using Lemma \ref{L:CombTree} below in the place of Lemma \ref{L:CombGen} to obtain a sequence $B_0\geq B_1 \geq B_2\geq \cdots$ in  $[s,D]$ so that each $B_{\ell}$ satisfies properties analogous to $(1)_{\ell}$ and $(2)_{\ell}$ from  the proof of  Lemma \ref{L:main}. The resulting $B:=\bigcap_{\ell}B_{\ell}$ will then satisfy the conclusion of Lemma \ref{L:mainTree}.

Let $X$ be a finite set and let $k,N\in\mathbb{N}$. A $(k,N)$--partition $P$ of  $X$ is any partition $P=\{X_0,\ldots,X_{k-1}\}$ of $X$ into $k$-many pieces so that  $|X_i|\geq N$ for all $i< k$. Let $\mathcal{P}(X)$ denote the power set of $X$. A {\bf section of $P$} is any $E\in\mathcal{P}(X)$ with $|E\cap X_i|\leq 1$ for all $i<k$.  Such a section $E$ is {\bf complete} if  $|E\cap X_i|= 1$ for all $i<k$.  We denote by $\mathcal{T}(P)$ and $\mathcal{CT}(P)$ the sets of all sections and all complete sections of $P$, respectively.

\begin{lemma}
\label{L:CombTree}
For every $k\in\mathbb{N}$ there is some $N\in\mathbb{N}$  so that  if  $P$ is a  $(k,N)$-partition on a finite  set  $X$ and $e\colon \mathcal{T}(P)\to \mathcal{P}(X)$ is any map, then there is some $E\in\mathcal{CT}(P)$ so that for all $F\subseteq E$ we have that either
\begin{equation}\label{EQ:E}
e(F)\subseteq F \; \mathrm{or} \; e(F)\not\subseteq E.
\end{equation}
\end{lemma}
\begin{proof}[Proof Sketch]
Let $k\in \mathbb{N}$. Notice that if there is $N\in\mathbb{N}$ such that the conclusion of the lemma holds for every 
$(k,N)$-partition $Q=\{Z_0,\ldots, Z_{k-1}\}$ with $|Z_i|=N$ for all $i<k$, then the conclusion of the lemma holds for all $(k,N)$-partitions. Indeed, let $P=\{X_0,\ldots, X_{k-1}\}$ be any $(k,N)$-partition and let $e\colon \mathcal{T}(P)\to \mathcal{P}(X)$ be any map. For each $i<k$, pick  $Z_i\subseteq X_i$ with $|Z_i|=N$, and set $Q:=\{Z_0,\ldots, Z_{k-1}\}$ and $Z:=\bigcup_i Z_i$. Consider the map $\widehat{e}\colon \mathcal{T}(Q)\to \mathcal{P}(Z)$ with $\widehat{e}(S)=e(S)\cap Z$. By assumption we may find  $E\in \mathcal{CT}(Q)$ so that for all $F\subseteq E$ we have $\widehat{e}(F)\subseteq F$ or $\widehat{e}(F)\not\subseteq E$. But then $E\in \mathcal{CT}(P)$, and for all $F\subseteq E$ we have  $e(F)\subseteq F$ or $e(F)\not\subseteq E$. 

So fix  $N\in \mathbb{N}$ and let $P=\{X_i: i<k\}$ be a partition of some set $X$ so that $|X_i|=N$ for all $i<k$. Let $e\colon \mathcal{T}(P)\to \mathcal{P}(X)$ be any map. We will show that as $N$ approaches $\infty$, the probability that a random $E\in \mathcal{CT}(P)$ satisfies (\ref{EQ:E}) for all $F\subseteq E$ tends to $1$. More precisely, let $\mathcal{B}$ be the collection of all pairs $(F,E)\in \mathcal{T}(P) \times \mathcal{CT}(P)$ for which (\ref{EQ:E}) fails. We claim that the  following ratio $R$ tends to $0$ as $N\to \infty$:
\[R:=\frac{|\{ E \in  \mathcal{CT}(P) \colon \exists F\in \mathcal{T}(P) \; (F,E)\in \mathcal{B}\}|}{|\mathcal{CT}(P)|}.\]
Notice  that the numerator of $R$ is at most $|\mathcal{B}|$. Hence, it suffices to show that   $|\mathcal{B}|/|\mathcal{CT}(P)|\to 0$ as  $N\to\infty$. This follows from the  inequalities:
\begin{align*}
\frac{|\mathcal{B}|}{|\mathcal{CT}(P)|}\leq  \frac{\sum_{m\leq k}{k \choose m} N^m\big(  N^{k-m}- (N-1)^{k-m} \big)}{N^k}\leq \\
\leq   \sum_{m\leq k} {k \choose m} \frac{N^k-(N-1)^k}{N^k}=2^k\big(1-\big((N-1)/N\big)^k\big).
\end{align*}
The second inequality  is clear. The first  inequality uses a simple counting argument, which we leave to the reader.
\end{proof}

\bibliographystyle{alpha}
\bibliography{bibliography}

\end{document}